\documentclass[12pt]{amsart}
\usepackage{latexsym,fancyhdr,amssymb,color,amsmath,amsthm,graphicx,listings,comment}
\usepackage[section]{placeins}
\newtheorem{thm}{Theorem} 
\newtheorem{lemma}{Lemma} 
\newtheorem{coro}{Corollary}
\let\paragraph\subsection

\title{Eigenvalue bounds of the Kirchhoff Laplacian}
\author{Oliver Knill}
\date{August 11, 2022, Updated May 21, 2024}
\address{Department of Mathematics \\ Harvard University \\ Cambridge, MA, 02138 }


\begin{document}
\maketitle

\begin{abstract}
We prove the inequality $\lambda_k \leq d_k+d_{k-1}$ for all the eigenvalues
$\lambda_1 \leq \lambda_2 \leq \cdots \leq \lambda_n$ 
of the Kirchhoff matrix $K$ of a finite simple graph or quiver with
vertex degrees $d_1 \leq d_2 \leq \cdots \leq d_n$ and assuming $d_0=0$. 
Without multiple connections, the inequality 
$\lambda_k \geq {\rm max}(0,d_k-(n-k))$ holds. 
A consequence in the finite simple graph or multi-graph case is that the 
pseudo determinant ${\rm Det}(K)$ counting the 
number of rooted spanning trees has an upper bound $2^n \prod_{k=1}^n d_k$
and that ${\rm det}(1+K)$ counting the number of rooted spanning forests 
has an upper bound $\prod_{k=1}^n (1+2d_k)$.
\end{abstract}

\section{The theorem}

\paragraph{}
Let $G=(V,E)$ be a {\bf quiver} with $n$ vertices. 
If it has no self-loops, this is a {\bf finite multi-graph}. Without
self-loops and no multiple-connections, this is a {\bf finite simple graph}. 
Denote by $\lambda_1 \leq \lambda_2 \leq \cdots \leq \lambda_n$ the 
{\bf ordered list of eigenvalues} of the {\bf Kirchhoff matrix} $K=B-A$, where $B$ is the
non-negative diagonal vertex degree matrix with {\bf ordered vertex degrees}
$d_1 \leq d_2 \leq \cdots \leq d_n$ and where $A$ is the non-positive  
{\bf adjacency matrix} of $G$.  

\paragraph{}
We assume $d_0=0$ so that $d_1+d_0=d_1$ and prove:

\begin{thm} $\lambda_k \leq d_k+d_{k-1}$, 
for all $1 \leq k \leq n$ and all quivers with $n$ vertices.  \label{1} \end{thm}

\paragraph{}
For finite simple graphs, the case $k=n$ is the {\bf spectral radius} estimate 
of Anderson and Morley \cite{AndersonMorley1985}. Their estimate generalizes to quivers:

\begin{lemma} $\lambda_n \leq d_n +d_{n-1}$ for all quivers with 
$n$ vertices. \end{lemma}

\paragraph{}
The details are in the proof section. The key is that we can still 
factor $K=F^T F$ and that the spectral radius of $K=K_0=F^T F$ agrees with the
spectral radius of $K_1=F F^T$ if $F$ is the incidence gradient matrix of
the quiver and because the spectral radius of $K_1$ is bounded
above by the maximal absolute column sum $d_n+d_{n-1}$ of $K_1$.

\paragraph{}
The case $k=1$ holds because 
$\lambda_1 \geq 0$ for quivers. The case $k=2$ follows from the Schur-Horn 
inequality \cite{Brouwer} as $\lambda_2 = \lambda_1+\lambda_2 \leq d_1+d_2$. 
Already $\lambda_3 \leq d_3+d_2$ goes beyond Schur-Horn. Any better estimate of the 
largest eigenvalue $\lambda_n$ in terms of $d_k$ would give also improved bounds on 
the general eigenvalue $\lambda_k$. Improvements like
\cite{GroneMerrisSunder1,GroneMerrisSunder2,Zhang2004,Das2004,
Guo2005,Shi2007,ShiuChan2009, LiShiuChan2010,ZhouXu}
all also would lead to analog bounds on all eigenvalues $\lambda_k$.

\paragraph{}
A lower bound $\lambda_k \geq 0$ holds for all quivers
because any matrix $K=F^T F$ has non-negative eigenvalues.
For graphs, a conjecture of Guo \cite{Guo2007} asked 
$\lambda_k \geq d_k-(n-k)+1$ for $k \geq 2$. This was proven in 
\cite{BrouwerHaemers2008} after some cases 
\cite{GroneMerrisSunder2,LiPan1999,Guo2007}. See also \cite{Brouwer} Proposition 3.10.2).  
The Brouwer-Haemers bound $\lambda_k \geq d_k-(n-k)+1$ has exceptions 
for $k=1$ like for $G=K_n$ \cite{BrouwerHaemers2008}. 
The weaker Brouwer-Haemers bound given in Theorem~2 works for all
quivers as long as no multiple connections are present.

\begin{thm} For multiple-connection free quivers,
 $\lambda_k \geq {\rm max}(0,d_k-(n-k))$ for all $1 \leq k \leq n$. \label{2} \end{thm}

\paragraph{}
Similarly as Lemma~1 kept a universal upper bound for $\lambda_n$, 
the lower bound only requires to control $\lambda_1$. Also the following
Lemma~2 is proven in the proof section. 

\begin{lemma} 
$\lambda_1 \geq d_1-(n-1)$ for all quivers without multiple connections. 
\end{lemma} 

\paragraph{}
This is sharp for the complete graph $K_n$ on $n$ vertices having no loops 
because $\lambda_1=0$ and $d_1=(n-1)$. For all
multi-graphs (quivers without loops), $\lambda_1=0$ with constant eigenvector. 
Lemma~2 does not generalize to quivers with multiple connections. An example is
the {\bf ribbon graph} $(V,E) =(\{1,2\},\{ (1,2),(1,2),(1,2)\})$, where
$d_1=d_2=3,n=2,\lambda_1=0,\lambda_2=6$, so that the left hand side
is $0$ but the right hand side is $3-(2-1)=2$. With $m$ strands in the ribbon,
we have $\lambda_1=0,\lambda_2=2m$ and $d_1=d_2=m$. The right hand side of 
Lemma~2 is now $m-1$ while the left hand side is either $0$ or $2m$. 

\paragraph{}
Also the proof of Lemma~2 is in the proof section. The main point is 
that because $\lambda_1 \geq 0$ holds for all quivers, we only need 
to cover $d_1>n-1$ which means $d_k>n-1$ for all $k$.
This implies without multiple connections that $l=d_k-(n-1)$ loops are present
at each node. But this allows to remove $l$ loops at each node to get
quiver eigenvalues $\lambda_k-l$ and degrees $d_k-l$ for all $k$. 
For a quiver without multiple connections and loops, the 
vertex degree node satisfies $d_1 \leq n-1$ and so $d_1-(n-1) \leq 0$.

\paragraph{}
It follows also from the expression $K=B-A$ with 
$B={\rm Diag}(d_1,\cdots d_n)$ positive semi-definite 
that $\lambda_k(K)=\lambda_k(B-A) \geq \lambda_k(-A)$ 
(\cite{HornJohnson2012} Corollary 4.3.12).
This {\bf Horn-Johnson lower bound} is valuable because it holds for all quivers.
It relies on the fact that $B$ is a positive semi-definite symmetric matrix. 
Horn-Johnson is somehow complementary to 
{\bf Brouwer-Haemers} which Theorem~2 extends to quivers, with a slight penalty
of 1 on the right hand side but which also makes it true for $k \geq 1$, while
the original Brouwer-Hamers bound holds for $k \geq 2$ for finite simple graphs. 
We see in the illustration section below that even for graphs,
like a bipartite or a grid graph, the two bounds can compete.

\paragraph{}
For a quiver $G=(V,E)$, the {\bf Euler characteristic} as a one-dimensional 
complex is $\chi(G)=|V|-|E|$. 
The union $V \cup E$ is technically not a {\bf finite abstract simplicial
complex} but a more general one-dimensional {\bf delta set}. Quivers could
actually be identified with delta sets of dimension $\leq 1$. The
{\bf cohomological Euler characteristic} is $b_0(G)-b_1(G)$, where 
$b_k(G)={\rm dim}({\rm ker}(K_k))$ with $K_0=F^T F$ and $K_1=F F^T$. 
The {\bf Euler-Poincar\'e formula} $\chi(G)=b_0(G)-b_1(G)$ can be 
proven as in \cite{Eckmann1944} or by using the heat flow given by 
the $(n+m) \times (n+m)$ matrix
$H=K_0 \oplus K_1$ using that $\chi(G)={\rm str}(e^{-tH})$ for all 
$t$ (it follows from ${\rm str}(H^k)=0$ for $k \geq 1$)
\cite{McKeanSinger,knillmckeansinger} and using that for $t=0$, 
${\rm str}(e^{-0 H})={\rm str}(1)=n-m=\chi(G)$ and that
for $t \to \infty$, the matrix $P=e^{-tH}$ is a projection 
onto harmonic forms and $\lim_{t \to \infty} {\rm str}(e^{-t H})=b_0(G)-b_1(G)$. 
This heat flow argument for Euler-Poincar\'e works in general for any delta set, 
also in dimension larger than $1$. 

\paragraph{}
There are several reasons why quivers are of interest in mathematics. 
The first is historical: the original {\bf K\"oenigsberg problem graph} \cite{Euler1736}
was a multi-graph (see e.g. \cite{Chartrand1984,Pappas1993,BiggsLloydWilson,Brualdi2004}).
In physics, quivers occur in the form of 
{\bf Feynman diagrams} or {\bf spin networks}, where the spin numbers attached
to the edges can be modeled as multiple connections. 
A {\bf category} in mathematics defines (after forgetting the associative
composition structure) a quiver, where the nodes are objects and the 
edges are {\bf morphisms} and the loops are the {\bf endomorphisms}. 
In chemical graph theory, {\bf molecules} can be
modeled as quivers; one can use loops to deal with weighted vertices and 
multiple edges can model stronger bonds like for example in $CO_2$, 
where oxygen $O$ has double bonds to carbon $C$ or a thiophene molecule
\cite{TrinajsticChemicalGraphTheory}.
Also {\bf Baker-Norine theory} \cite{BakerNorine2007} for {\bf discrete Riemann-Roch} 
considers multi-graphs; {\bf divisor} are in that frame work integers attached 
to a vertex which at least for {\bf essential divisors} can be modeled 
as loops. Finally, any {\bf magnetic Schr\"odinger operator} on a quiver with 
non-negative real magnetic potential and non-negative scalar potential 
satisfies the bounds of Theorem~1 because these operators 
are approximated by scaled versions of Kirchhoff matrices of quivers. 
The simplest magnetic Schr\"odinger operator is a {\bf periodic Jacobi matrix}
$(L u)_k = -a_k u_{k+1} -a_{k-1} u_{k-1} + (a_k+a_{k-1}+d_k) u_k$ 
which for integer $a_k=a_{k+n} \geq 0,d_k=d_{k+n} \geq 0$  is
the Kirchhoff matrix of a quiver 
for which the underlying finite simple graph is the cyclic graph $C_n$. 
The number $a_k$ on the edge $(k,k+1)$ counts the multiplicity of the
edge forming an ``edge amplitude". The $d_k$ loops together with the
$a_k + a_{k-1}$ connections give a ``vertex amplitude". 

\paragraph{}
Already the Corollary $\lambda_k \leq 2d_k$ of Theorem~(\ref{1}) is stronger 
than what the {\bf Gershgorin circle theorem} 
\cite{Gershgorin,GershgorinAndHisCircles}
gives in this case: the circle theorem assures in the self-adjoint case
that in every interval $[0,2d_k]$, there is at least one eigenvalue 
$\lambda_l$ of $K$. But it does not need to 
be the k'th one. In the Kirchhoff case, the Gershgorin circles are nested. For
finite simple graphs, $0$ is always in the spectrum. Theorem~(\ref{1}) gives more information. 
The spectral data $\lambda_1=0,\lambda_2=10,\lambda_3=10$ for example would be
Gershgorin compatible to $d_1=1, d_2=3, d_3=7$ because there is 
an eigenvalue in each closed ball $[0,2], [0,6]$ and $[0,14]$. 
But these spectral data do not occur in any quiver and so not in a 
finite simple graph, as they would contradict Theorem~(\ref{1}) which 
is incompatible with $\lambda_2 = 2 d_2+4$. 
Theorem~\ref{1} keeps the eigenvalues more aligned with the degree sequence, 
similarly as the {\bf Schur-Horn theorem} does.

\paragraph{}
An application for connected finite simple graphs (which have one eigenvalue $\lambda_1=0$)
is that the {\bf pseudo determinant} ${\rm Det}(K)= \prod_{k=2}^n \lambda_k$, 
which by the Kirchhoff-Cayley {\bf matrix tree theorem} counts
the number of {\bf rooted spanning trees}, has an upper bound $2^n \prod_{k=1}^n d_k$
and that ${\rm det}(1+K)$ which by the Chebotarev-Shamis {\bf matrix forest theorem}
counts the number of {\bf rooted spanning forests}, has an upper bound 
$\prod_{k=1}^n (1+2d_k)$.  
These determinant inequalities do not follow from Gershgorin, nor from the 
Schur-Horn inequalities. 

\paragraph{}
We like to rephrase the determinant
inequalities as bounds on the {\bf spectral potential}
$U(z)=(1/n) \log \det(K-z) \leq 2 + (1/n) \log\det(M-z)$,
where $M={\rm Diag}(d_1,\dots,d_n)$ is the diagonal matrix and 
$z \leq 0$ is real. We made use of Theorem~(\ref{1}) to show that 
the potential $U(z)$ has a Barycentric limit. The Barycentric refinement
$G_1$ of a finite simple graph $G$ takes the complete subgraphs as vertices of a new 
graph $G_1$ and connects two if one is contained in the other. The spectrum of
any locally defined Laplacian like the 0-form Laplacian $K_0$ has then a Barycentric limit in the sense
that the density of states measure $dk$ converges weakly to a measure that only depends
on the dimension. In one dimension, where the limiting $dk$ has compact support,
it is the absolutely continuous equilibrium measure on $[0,4]$ but in higher dimensions
it appears to be more complicated and has unbounded (probably fractal) support.
Our estimate relates the potential $U$ of the interacting system with the potential 
$(1/n) \log\det(M-z)$ of the {\bf non-interacting system}, where $M$ can be thought 
of as the Kirchhoff matrix of a ``non-interacting" system, where we have $d_k$ 
self-loops at vertices $k$, even-so the tree 
and forest interpretations do not apply for diagonal matrices.

\section{Quivers}

\paragraph{}
Quivers extend multi-graphs in that additionally to multiple connections,
also self-loops are allowed. For us here, the direction of the edges is not important. 
Usually, all edges come with a direction. But this is just a {\bf choice of basis}. 
It is irrelevant to the spectral discussion because orientations do not 
enter the Kirchhoff matrix $K=F^TF$, which is the $0$-form Laplacian. 
The orientations do enter the $1$-form Laplacian $F F^T$ because $1$-forms live 
on the edges. Different orientations just produce similar and so 
isospectral 1-form Laplacians.

\paragraph{}
Extending Theorem~1 and Theorem~2 from finite simple graphs to quivers 
simplified the proof because the class of Kirchhoff matrices of quivers is 
invariant under the process of taking principal sub-matrices. It is not uncommon
in induction proofs that more general proofs are easier because the induction assumption
becomes stronger too. 
The induction step would already fail on the class of {\bf multi graphs}
which are graphs with possible multiple connections but which have no loops.
The notations in the literature are not always the same so that 
we will now repeat the definition of quiver we use. 
\footnote{Quiver-graphs have been called ``graphs" \cite{BR},``general graphs" \cite{Brualdi2004},
``pseudo graphs" \cite{Vasudev}, ``loop multi-graphs" \cite{TrinajsticChemicalGraphTheory} 
or ``r\'eseaux" \cite{VerdiereGraphSpectra}.}

\paragraph{}
A {\bf quiver} $G=(V,E)$ is defined by a finite set of vertices $V=\{1,\dots, n\}$
and a finite list of edges $E=\{ (v_j,w_j) \in V \times V, j=1,\dots,m \}$, 
where no restrictions are imposed on the list $E \subset V \times V$. 
The edges are given by ordered lists $(a,b)$ providing an {\bf orientation}. As was pointed out
already, this is a choice of basis when writing down the gradient matrix $F$, but 
it is irrelevant for spectral purposes. 
Several copies of the same edge can appear. They are called {\bf multiple connections}.
Additionally, {\bf self-loops} $(v,v)$ are allowed. The data of an undirected 
quiver can also be given by a {\bf finite simple graph} $H=(V,E)$ and two functions 
$X: V \to \mathbb{N}$ and $Y: E \to \mathbb{N}$ telling how many loops there are at 
each node and how many additional connections each edge 
in $F$ has. \footnote{One would call $X$ an ``effective divisor" in the discrete 
Riemann-Roch terminology \cite{BakerNorine2007} because they are non-negative.
Self-loops are treated here as edges. They do not contribute to the 
{\bf adjacency matrix} $A$ but to the {\bf degree matrix} $B$ in $K=B-A$. 
There are $|V|=n$ vertices and $|E|=m$ edges and the later include the loops.} 
If the edges connecting different nodes are oriented, one can also indicate this
by talking about a {\bf directed multi-graph} or {\bf directed quiver}. 
We use (an arbitrary) orientation on edges only for defining
the {\bf gradient matrix} $F$. As for graphs, this orientation does not affect the 
Kirchhoff matrix $K=F^T F$ (the 0-form Laplacian) nor the spectrum of the matrix 
$K_1=F F^T$ (the 1-form Laplacian).

\paragraph{}
The {\bf adjacency matrix} $A$ of $G$ is the symmetric $n \times n$ matrix with 
$A_{ij}=| \{(v,w) \in V^2,  v \neq w, (v,w) \in E \; {\rm rsp} \; (w,v) \in E \} |$,
the number of edges hitting $v$. Note that loops do not enter the adjacency matrix.
The {\bf vertex degree matrix} is defined as a diagonal matrix with 
$B_{ii} = | \{ (v,w) \in V^2, v = v_i \; {\rm or} \; w = v_i \} |$ counting the number of
edges that entering or leaving. But each loop just counts as $1$ in a diagonal entry of $B$. 
The {\bf Kirchhoff matrix} of the $G$ is then defined as $K=B-A$. 
For the 4-loop clover graph (see Example~2), we have $A=[0],B=[4]$ so that 
$K=K_0=B-A=[4]-[0]=[4]$. 
For $G=\{\{1\},\{\}\}$, one has $B=[0]$ and $K=B-A=[0]-[0]=[0]$ and the 1-form Laplacian 
is the {\bf empty matrix}, as $G$ is 0-dimensional. For the {\bf Kronecker quiver}
$G=\{ \{1,2\},\{ (1,1),(2,2),(1,2),(1,2)\} \}$ with 2 vertices and 4 edges, we have 
$K=\left[ \begin{array}{cc} 3 & -2 \\ -2 & 3 \\ \end{array} \right]$ with eigenvalues $1,5$. 
The Kirchhoff matrix appears in \cite{BakerNorine2007} for multi-graphs, 
meaning quivers without self-loops.

\paragraph{}
As in the case of finite simple graphs, there is also for quivers 
a factorization $K= F^T F$ of the Kirchhoff matrix $K$. 
The $m \times n$ matrix $F$ is the {\bf discrete gradient}, mapping a function on vertices
(0-form) to a function on edges (1-form) and $F^T$ is the 
{\bf discrete divergence}, mapping functions on edges to functions on vertices. 
\footnote{We usually would write $d_0$ instead of $F$ for the gradient and $d_0^*$ 
instead of $F^T$ for the divergence as in analysis, but $d_k$ are used for vertex 
degrees here.}
We have $F_{(a,b),v} = 1$ if $v=a=b$; if $a \neq b$, then 
$F_{(a,b),b} = 1$ and $F_{(a,b),a}=-1$ and $F_{(a,b),v}=0$ if $v \notin \{a,b\}$. 
The matrix $K=K_0$ is the {\bf 0-form Laplacian} while $K_1=F F^T$ is the 
{\bf 1-form Laplacian}. One can build a $(n+m \times n+m)$ Dirac matrix $D=d+d^*$ with 
$d^2=0$ and form $H=K_0 \oplus K_1=D^2=(d+d)^* = d d^* + d^* d$. The fact that $K_0,K_1$
are essentially isospectral is then a special case of the McKean-Singer symmetry. 
\footnote{ Henry McKean, one of the authors of \cite{McKeanSinger}, 
is acknowledged in Anderson-Morely \cite{AndersonMorley1985} as suggesting the study. }

\paragraph{}
{\bf Example 1:} If $V=\{1,2,3\}$ and 
$E=\{ (1,1),(1,1),(1,2),(1,2),(1,2),(2,1),(2,2),(2,3) \}$, 
\footnote{To illustrate the role orientation, we added an $(2,1)$ edge 
with a different orientation than $(1,2)$. 
This is a choice of basis or a choice of gauge. It affects gradient $F$, 
divergence $F^T$ and 1-form Laplacian $K_1=F F^T$ and harmonic 
1-forms (vectors in the kernel of $K_1$). It does not affect the 0-form 
Laplacian $K=FF^T$. }
$$ K = \left[ \begin{array}{ccc}
                  6 & -4 & 0  \\
                  -4 & 6 & -1 \\
                  0 & -1 & 1  \\
                 \end{array} \right]
 = F^T F = \left[
\begin{array}{cccccccc}
 1 & 1 & 1 & 1 & 1 & -1 & 0 & 0   \\
 0 & 0 & -1 & -1 & -1 & 1 & 1 & 1 \\
 0 & 0 & 0 & 0 & 0 & 0 & 0 & -1   \\
\end{array} \right] 
\left[ \begin{array}{ccc}
 1 & 0 & 0  \\
 1 & 0 & 0  \\
 1 & -1 & 0 \\
 1 & -1 & 0 \\
 1 & -1 & 0 \\
 -1 & 1 & 0 \\
 0 & 1 & 0  \\
 0 & 1 & -1 \\
\end{array} \right] \; . $$
While $F, F^T$ and $K_1=F F^T$ depend on the given orientation of the edges, 
the Kirchhoff matrix $K=F^T F$ does not. Every diagonal entry $K_{i,i}$ is 
a {\bf vertex degree}, counting the number of edges leaving or entering the node $v_i$, 
where a loop counts as one edge. The $d_i$ provide an {\bf ordered
list} of these diagonal entries. We could renumber the vertices so that $d_i=K_{ii}$ 
but this is not necessary. The entry $-K_{i,j}$ for $i \neq j$ counts the number of 
connections from the vertex $v_i$ to the vertex $v_j$. 

\paragraph{}
{\bf Example 2:} The {\bf clover graph} is $G=(V,E)=(\{\{1\},\{(1,1),(1,1),(1,1),(1,1)\})$ has
one vertex and 4 loops. Now $F^T=[1,1,1,1]$ and $K=K_0=F^T F$ is the $1 \times 1$ 
matrix $[4]$  and the 1-form Laplacian is
$K_1=F F^T=\left[ \begin{array}{cccc}1&1&1&1 \\1&1&1&1 \\1&1&1&1 \\1&1&1&1 \end{array} \right]$
is a $4 \times 4$ matrix. 
Theorems~1 and 2 are both sharp in this example because $d_1=4, \lambda_1=4, n=1$.
This happens for any $m$-loop graph with $n=1$ vertex. 
The {\bf Betti numbers} $b_0={\rm dim}({\rm ker}(K_0))=0$
and $b_1={\rm dim}({\rm ker}(K_1))=n-1$ satisfy the {\bf Euler-Poincar\'e relations}
$\chi(G)=b_0-b_1=|V|-|E|$ for the {\bf Euler characteristic} as a $1$-dimensional 
cell complex. It is already not a simplicial complex but a {\bf delta set}. 
Note that $K$ is invertible, leading to $b_0=0$. 
The {\bf one loop graph} $G=(V,E)=(\{\{1\},\{(1,1)\})$ has trivial cohomology 
$b_0=b_1=0$ and $|V|=|E|=1$. 

\begin{figure}[!htpb]
\scalebox{0.5}{\includegraphics{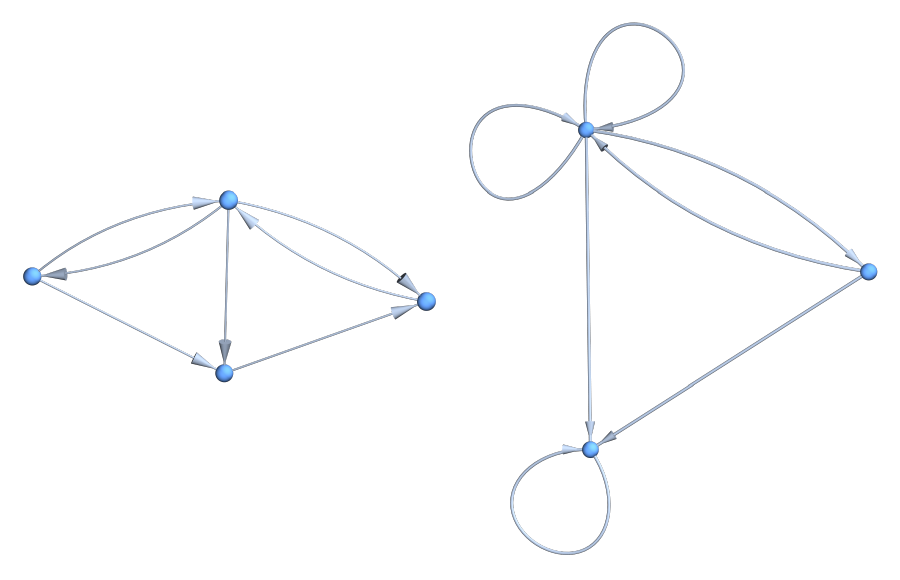}}
\label{Koenigsberg}
\caption{
The {\bf K\"onigsberg bridge graph} $G$ considered by Leonhard Euler 
\cite{Euler1736} in 1736 is a multi-graph of Euler characteristic $|V|-|E|=4-6=-2$
(when seen as a 1-dimensional cell complex).
Removing the first row and column of the Kirchhoff matrix
is the Kirchhoff matrix of a quiver with $3$ vertices.
}
\end{figure}

\paragraph{}
{\bf Example 3:} The {\bf K\"onigsberg bridge graph}
$$ G = (V,E) = ( \{ 1,2,3,4 \}, \{ (12),(21),(41),(14),(23),(34),(13) \}) $$
is historically the earliest appearance of a multi-graph in mathematics. 
The gradient matrix is $F=\left[
\begin{array}{cccc}
 1 & -1 & 0 & 0 \\
 -1 & 1 & 0 & 0 \\
 -1 & 0 & 0 & 1 \\
 1 & 0 & 0 & -1 \\
 0 & 1 & -1 & 0 \\
 0 & 0 & 1 & -1 \\
 1 & 0 & -1 & 0 \\
\end{array} \right]$. We have chosen a variety of orientations to illustrate
how it affects the entries. The Kirchhoff matrix is 
$K=F^T F=\left[ \begin{array}{cccc}
 5 & -2 & -1 & -2 \\
 -2 & 3 & -1 & 0 \\
 -1 & -1 & 3 & -1 \\
 -2 & 0 & -1 & 3 \\
\end{array} \right]$. We have 
$d_1=d_2=d_3=3,d_4=5$ and $\lambda_1=0,\lambda_2=3,\lambda_3=4,
\lambda_5=7$. The principal sub-matrix in which the first row and
column are deleted, has the eigenvalues $3-\sqrt{2},3,3+\sqrt{2}$.
The Euler characteristic is $\chi(G)=|V|-|E|=4-7=-3$. The spectrum 
of the 1-form Laplacian
$K_1 = F F^T = \left[ \begin{array}{ccccccc}
 2 & -2 & -1 & 1 & -1 & 0 & 1 \\
 -2 & 2 & 1 & -1 & 1 & 0 & -1 \\
 -1 & 1 & 2 & -2 & 0 & -1 & -1 \\
 1 & -1 & -2 & 2 & 0 & 1 & 1 \\
 -1 & 1 & 0 & 0 & 2 & -1 & 1 \\
 0 & 0 & -1 & 1 & -1 & 2 & -1 \\
 1 & -1 & -1 & 1 & 1 & -1 & 2 \\
\end{array} \right]$ is $(0,0,0,0,3,4,7)$. The cohomological 
Euler characteristic is 
$b_0-b_1={\rm dim}({\rm ker}(L_0))-{\rm dim}({\rm ker}(L_1)) =1-4=-3$. 

\paragraph{}
{\bf Example 4:} The {\bf Good-Will-Hunting multi-graph}  \cite{MathMovies,HorvathKorandiSzabo}
$$  G=(\{ \{1,2,3,4 \}, \{ (12),(24),(14),(23),(23) \}) $$
was seen as a challenge on a chalk board in the 1997 coming-of-age 
movie ``Good Will Hunting". Its Euler characteristic as a 1-dimensional cell complex
is $\chi(G)=|V|-|E|=-1$. We have 
$F=\left[
\begin{array}{cccc}
 1 & -1 & 0 & 0 \\
 0 & 1 & 0 & -1 \\
 1 & 0 & 0 & -1 \\
 0 & 1 & -1 & 0 \\
 0 & 1 & -1 & 0 \\
\end{array}
\right]$ and
$K=F^T F= B-A = \left[ \begin{array}{cccc}
 2 & -1 & 0 & -1 \\
 -1 & 4 & -2 & -1 \\
 0 & -2 & 2 & 0 \\
 -1 & -1 & 0 & 2 \\
\end{array} \right]$. In the movie, the first competition assignment written on the 
MIT chalk-board in the Good-Will-Hunting problem was the task to
write down the adjacency matrix $A$. We have $d_1=d_2=d_3=2, d_4=4$ and 
$\lambda_1=0,\lambda_2=(7-\sqrt{17})/2,\lambda_3=3,\lambda_4=(7+\sqrt{17})/2$. 
The principal sub-matrix in which the second row and column were deleted, has eigenvalues
$1,2,3$. It is the Kirchhoff matrix of the quiver 
$G=(\{ \{1,3,4\}, \{ (14),(11),(44),(33),(33) \})$. The $1$-form Laplacian is
a $5 \times 5$ matrix with a 2-dimensional kernel. A basis for these {\bf harmonic 1-forms}
are supported on a loop of length 3 and a loop of length 2. Since $b_0=1,b_1=2$, 
we again get $\chi(G)=1-2=-1$ for the Euler characteristic. 

\paragraph{}
{\bf Example 5:} The {\bf Kronecker quiver} $Q=\{ \{1,2\},\{ (1,1),(2,2),(1,2),(1,2) \} \}$
has the gradient 
$F = \left[ \begin{array}{cc}
                   1 & 0 \\
                   0 & 1 \\
                   1 & -1 \\
                   1 & -1 \\
                  \end{array} \right]$ and 
$K=F^T F=\left[ \begin{array}{cc} 3 & -2 \\ -2 & 3 \\ \end{array} \right]$ with eigenvalues $1,5$
and $K_1=F F^T=\left[
                  \begin{array}{cccc}
                   1 & 0 & 1 & 1 \\
                   0 & 1 & -1 & -1 \\
                   1 & -1 & 2 & 2 \\
                   1 & -1 & 2 & 2 \\
                  \end{array}
                  \right]$ and spectrum $\{0,0,1,5\}$. The Euler characteristic is $\chi(Q)=|V|-|E|=2-4$
or $\chi(Q)=b_0-b_1=0-2$. If the edges are interpreted as morphisms or endomorphisms and the
associative composition law is added, $Q$ is a {\bf category}. A general quiver is a functor $Q \to {\rm Set}$
so that the {\bf quiver category} is a category of {\bf presheaves} on the opposite category $Q^{op}
= Q=\{ \{1,2\},\{ (1,1),(2,2),(2,1),(2,1) \} \}$. Quivers are delta sets of maximal dimension $1$. 

\begin{figure}[!htpb]
\scalebox{0.9}{\includegraphics{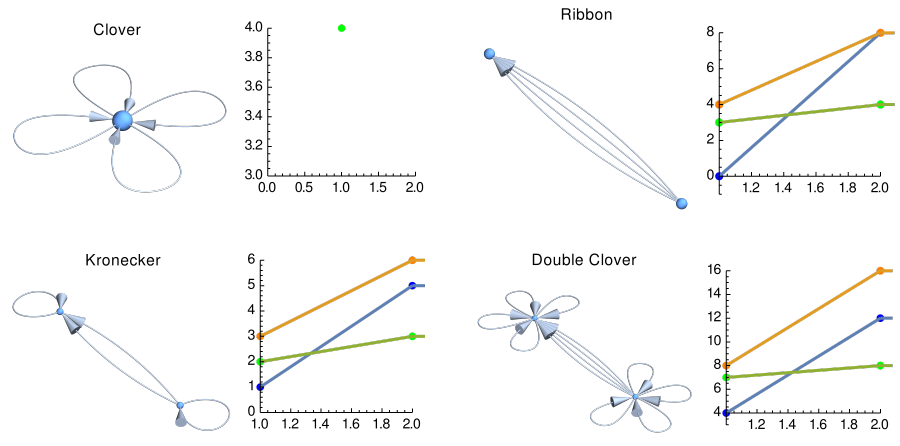}}
\label{CloverKroneckerRibbon}
\caption{
For the clover, with $K=[m]$, Theorems~1,2
are sharp. For graphs with multiple connections like the ribbon, the lower
bound Theorem~2 can fail. 
If in the Kronecker quiver $Q=\{ V=\{V,E\}, E=\{ (V,V),(E,E),(V,E),(V,E)\} \}$,
edges are interpreted as morphisms, (loops are endomorphisms), and a composition is
added, one has a finite category with 2 objects. A quiver can be
seen as a functor from $Q$ to the category of finite sets. Seen as such, 
quivers are a {\bf finite topos}. 
}
\end{figure}

\begin{figure}[!htpb]
\scalebox{0.5}{\includegraphics{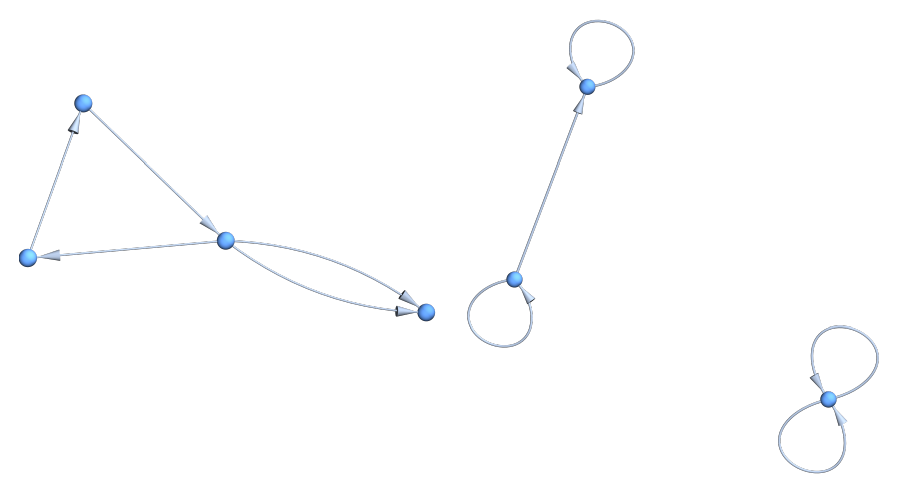}}
\label{Spectra}
\caption{
The picture shows the {\bf Good-Will-Hunting quiver} and the principal 
sub-quiver obtained by deleting the vertex $v$ with maximal degree, 
snapping all connections $(v,w)$ to $v$ to loops $(w,w)$
on its neighbors $w$. 
}
\end{figure}

\section{Proofs}

\paragraph{}
Induction is enabled by the following observation.
\footnote{An earlier version of this paper formulated Theorem~1 for finite simple graphs only
and artificially introduced a class $\mathcal{K}$ of matrices, 
that is invariant under principal sub-matrix operation. 
It is contained in the class of Kirchhoff matrices of quivers.}

\begin{lemma} Any principal sub-matrix of a Kirchhoff matrix of quiver
with $2$ or more vertices is the Kirchhoff matrix of a quiver. \end{lemma} 
\begin{proof}
Removing the row and column to a vertex $v$ has the effect that we look at the
quiver with vertex set $V \setminus \{v\}$ for which $k$ additional loops
have been added at every entry $w \in V \setminus \{v\}$ which had $k$ multiple
connections to $v$ in the original graph $G$. We could also allow $n=1$ (in which
case we have the $1 \times1$ matrix $K=[l]$ for a quiver with $l$ loops),
but then the principal sub-matrix is the {\bf empty matrix} which can be thought of
as the Kirchhoff matrix of the {\bf empty graph} $\{ \emptyset, \emptyset \}$ 
which is the initial object in the category. 
\end{proof}

\paragraph{}
For the lower bound, we will need the following observation:

\begin{lemma} Any principal sub-matrix of a Kirchhoff matrix of quiver without 
multiple connections with $2$ or more vertices is the Kirchhoff matrix of a quiver without 
multiple connections. \end{lemma}   
\begin{proof}
We can repeat the proof from the previous lemma and note that removing the row and column 
to a vertex $v$ does not introduce multiple connections and so keep us in the 
class of quivers without multiple connections. 
\end{proof} 

\paragraph{}
A consequence is that Kirchhoff matrices of quivers or Kirchhoff matrices of quivers without
multiple connections form classes of matrices which are invariant under the operation of 
taking principal sub-matrices. We should also note the obvious:

\begin{lemma}
The Kirchhoff matrix of any quiver has only non-negative eigenvalues.
\end{lemma}

\begin{proof}
This is a general linear algebra result: use 
$K = F^T F$ and write $(v,w)$ for the dot product of two vectors $v,w$. 
If $K v = \lambda v$ with unit eigenvector $v$, we have 
$\lambda = \lambda (v,v) = (\lambda v,v) = (K v,v) = (F^T F v,v) = (Fv,Fv) \geq 0$. 
\end{proof} 

\paragraph{}
The proof of Theorem~1 and Theorem~2 both use induction with respect to the number 
$n$ of vertices. The induction step uses the {\bf Cauchy interlace theorem} 
(also known as {\bf separation theorem}) \cite{HornJohnson2012}:
if $\mu_k$ are the eigenvalues of a principal $(n-1 \times n-1)$ sub-matrix of $K$, 
then $\lambda_k \leq \mu_k \leq \lambda_{k+1}$.
(The Cauchy interlace theorem follows from the {\bf Hermite-Kakeya theorem} 
for real polynomials \cite{Hwang2004,Fisk2005,RahmanSchmeisser}. 
If $f$ is a monic polynomial of degree $n$ with real roots
$\lambda_1 \leq \cdots \leq \lambda_n$
and $g$ is a monic polynomial of degree $n-1$ with real roots
$\mu_1 \leq \cdots \leq \mu_{n-1}$ then
{\bf $g$ interlaces $f$} if
$\lambda_1 \leq \mu_1 \leq \lambda_2 \leq \cdots \leq \mu_{n-1} \leq \lambda_n$.)
For an other short proof of the interlace theorem, see 
\cite{Nica2018} (Theorem 10.6).

\paragraph{}
Here is the proof of Theorem~1: 

\begin{proof} 
The induction foundation $n=1$ works because the result holds for the
Kirchhoff matrix of a quiver with $n=1$ vertices and $l$ loops.
Indeed, this $1 \times 1$ matrix $K=[l]$ which has one non-negative entry $l$ and
$\lambda_1=l$ and $d_1=l$. \\
Assume now the claim is true for all quivers with $n$ vertices. 
Take a quiver with $n+1$ vertices and let $K$ denote its Kirchhoff matrix. 
Pick a row with maximal diagonal entry and delete this row and the corresponding 
column intersecting in that diagonal entry. Using the Lemma, it
produces a Kirchhoff matrix $L$ of a quiver $H$ with $n$ vertices.  By
Cauchy interlacing, the eigenvalues $\mu_k$ of $L$ satisfy 
$0 \leq \lambda_1 \leq \mu_1 \leq \cdots \leq \lambda_n \leq \mu_n \leq \lambda_{n+1}$. 
We know by induction that $\lambda_k \leq \mu_k \leq d_k + d_{k-1}$ for $1 \leq k \leq n$. 
So, also $\lambda_k \leq \mu_k$ for $k \leq n$. 
The interlace theorem does not catch yet the largest eigenvalue $\lambda_{n+1}$. 
This requires an upper bound for the {\bf spectral radius}. 
This is where Anderson and Morley \cite{AndersonMorley1985} come in.
They realized in 1985 that $K=F^T F$ is essentially iso-spectral to $K_1=F F^T$,
meaning that the non-zero eigenvalues agree. 
The later is the adjacency matrix of the {\bf line graph} but with the modification
that all diagonal entries are $2$ (for edges connecting two vertices) or $1$
(for loops). Since the row or column sum
of $F F^*$ is bound by $d_k(a) + d_k(b)$ for any edge $(a,b)$, the upper 
bound holds. The Anderson-Morley bound for quivers adds the last inequality
$\lambda_{n+1} \leq d_{n+1} + d_n$ for all quivers with $n+1$ vertices. 
The induction proof is complete. 
\end{proof} 

\paragraph{}
Here is the proof of Theorem~2: 

\begin{proof} 
The induction foundation $n=1$ holds because every quiver without multiple 
connections satisfies $d_1 \leq n-1$. This implies $0=\lambda_1 \geq d_1-(n-1)$.
Assume that the claim is true for all quivers with $n$ vertices. 
Take a quiver with $n+1$ vertices and let $K$ denote its Kirchhoff matrix. 
Pick the row with the minimal diagonal entry and delete this row and the corresponding 
column. This principal sub-matrix is a Kirchhoff matrix $L$ of a quiver $H$ 
with $n$ vertices. By Cauchy interlacing and Lemma~2, the eigenvalues $\mu_k$ of $L$ satisfy 
$d_1-((n+1)-1) \leq \lambda_1 \leq \mu_1 \leq \cdots \leq \lambda_n \leq \mu_n \leq \lambda_{n+1}$. 
From this and the induction assumption, we conclude
$\lambda_k \geq \mu_{k-1} \geq d_{H,k-1} - (n_H-(k-1))$ 
for $k \geq 2$. But $d_{H,k-1}=d_{G,k}$ and $n_H=d_G-1$ so that
$\lambda_{k} \geq d_{H,k-1} - (n_H-(k-1)) = d_{G,k} - (n_G-1-(k-1)) = d_{G,k}-(n_G-k)$
which needed to be shown. 
\end{proof} 

\paragraph{}
Here is the proof of Lemma~1:

\begin{proof}
The proof is the same as in the Anderson-Morley case, because we still
can relate the spectrum of $K=F^T F$ with the spectrum of $K_1 = F F^T$.
The later is a $m \times m$ matrix which has $1$ in the diagonal if the edge had
been a loop and $2$ in the diagonal, where the edge connects two points. Furthermore,
$K_1(e,f)=-1$ or $1$. If we look at the $k$'th column belonging to an edge $e$,
then the sum of the absolute values of the column entries is bound by $d_n + d_{n-1}$. 
The spectral radius $K_1$ is therefore bound by $d_n + d_{n-1}$. Since $K$ and $K_1$
are essentially iso-spectral, meaning that they have the same non-zero eigenvalues, 
also the spectral radius of $K$ is bounded above by $d_n + d_{n-1}$. 
\end{proof} 

\paragraph{}
Here is the proof of Lemma~2: 

\begin{proof}
We use induction for fixed $n$ with respect to the number $l$ of self-loops at the lowest
degree vertex. The statement is true for graphs for $l=0$ loops (multi-graphs),
because there are then maximally $n-1$ neighboring vertices $d_1 \leq n-1$ so that 
$d_1 - (n-1) \leq 0$.
The general statement $\lambda_1 \geq 0$ proven above implies
now that if we have $l$ loops at the lowest degree vertex, then
$0 \leq \lambda_1 < d_1-(n-1)$. This means $d_1>n-1$ so that there are at 
least $q=d_1-(n-1)>0$ loops at this vertex. Because $d_k-(n-1) \geq d_1-(n-1)$,
there are also at least $q$ loops at every other vertex. Removing one loop at
every vertex lowers all eigenvalues by $1$ and all $d_k$ by $1$ and $l$ by $1$. 
We get a quiver with $l-1$ loops at the lowest vertex degree vertex, 
where still $d_1 - (n-1) >\lambda_1$. This is satisfied because the 
induction assumption was $d_1-(n-1) \leq \lambda_1$ for all quivers with 
$l-1$ or less loops at the vertex with the lowest degree.
\end{proof} 

\section{Remarks}

\paragraph{}
Quiver Kirchhoff matrices allow to consider the {\bf Schr\"odinger case} 
$L = K + V$, where $V$ is a diagonal 
matrix containing non-negative potential values $V_k$. More generally, we can use
{\bf electro-magnetic operators} 
$L u(v) = (K u)(v) + V(v) u(v) - \sum_{(v,w) \in E} W(w)$ with 
$W(w) \geq 0$ and $V(v) \geq 0$. 
If $V_k,W_k \in \mathbb{N}$, this is a situation which is covered by quivers as the
$V_k$ can be seen as loops and $W_k$ as a multiplicity of edges.
\footnote{\cite{VerdiereGraphSpectra} uses the term ``differential operator" or
``operator with magnetic field".}

\paragraph{}
While technically, we are bound to non-negative integer-valued potentials $V,W$
a scaling and translation of $L$ allows to cover rational-valued potentials. 
By a limiting procedure one can get to the real case. 
The upshot is that we can think about Kirchhoff Laplacians already as electromagnetic
Laplacians with electric field modeled by loops and magnetic field modeled by 
multiple connections. 

\begin{coro}[Schroedinger estimate]
If $L$ is a electro-magnetic operator, where $K$ is the Kirchhoff Laplacian of a quiver,
$V$ is a non-positive real electric field on vertices and $W$ a non-negative real 
magnetic field on edges, then $\lambda_k \leq d_k+d_{k-1}$, where $d_k$ are
the diagonal entries of $K$, ordered $d_1 \leq d_2 \leq \cdots \leq d_n$. 
\end{coro}

\paragraph{}
Note that in the electric case $L=K+V$, the spectrum monotonically increases when
increasing $V$ and that in the magnetic case $L=K-W$, the spectrum monotonically 
increases when increasing $W$. We can see the change of eigenvalues through deformation:
just track an eigenvalue $\lambda$ with eigenvector $v$
if a diagonal entry $V_k$ is varied in a {\bf rank-one perturbation}.
The {\bf first Hadamard deformation formula} gives $\lambda'(t) = v_k^2 \geq 0$.
The {\bf second Hadamard deformation formula} gets $\lambda''(t)>0$ depending on the
other eigenvalues illustrating {\bf eigenvalue repulsion}. Increasing $V$ 
increases all eigenvalues. If we increase the magnetic field $W$, which means by
our sign choice counts like reducing the number of bonds, the spectrum decreases.
If we decrease the electric field, which means reducing the number of loops, the
spectrum decreases again. This is the reason for the Corollary. 

\paragraph{}
Similarly than the Schur-Horn inequality
$\sum_{j=1}^k \lambda_j \leq \sum_{j=1}^k d_j$ 
which is true for any symmetric matrix with diagonal entries $d_j$, 
Theorem~(\ref{1}) controls how close the {\bf ordered eigenvalue sequence}
is to the {\bf ordered vertex degree sequence}. Theorem~(\ref{1}) has a different
nature than Schur-Horn:
for example, if the eigenvalues increase exponentially like $\lambda_k = 6^k$, then 
the inequality $\lambda_k \leq 2 d_k$ (which follows from Theorem~1, implies 
on a logarithmic scale $\log(\lambda_k) \leq \log(2) + \log(d_k)$. Schur-Horn does
not provide that. Also the {\bf Gershgorin circle theorem}
would only establish $\log(\lambda_k) \leq \log(2) + \log(d_n)$, 
where $d_n$ is the largest vertex degree entry, because that theorem assures only that in each 
Gershgorin circle, there is at least one eigenvalue.

\paragraph{}
Anderson and Morley \cite{AndersonMorley1985} already had the better bound 
$\max_{(x,y) \in E} d(x) + d(y)$ and Theorem~(\ref{1}) could be improved in that
$d_k+d_{k-1}$ could be replaced by ${\rm max}_{j, (x_j,x_k) \in E}  d_k+d_j$.  
Any general better upper bound of the spectral radius would lead to sharper results.
The Anderson-Morley estimate as an early use of a {\bf McKean-Singer super symmetry}
\cite{McKeanSinger} (see \cite{knillmckeansinger} in the discrete),
in a simple case, where one has only $0$-forms and $1$-forms. There, it reduces to
the statement that $K=F^* F$ is essentially isospectral to $F F^*$ which is true for 
all matrices. It uses that the Laplacian $K$ is of the form $F^* F = {\rm div} {\rm grad}$,
where $F={\rm grad}$ is the {\rm incidence matrix} $F f( (a,b) ) = f(b)-f(a)$ for functions
$f$ on vertices (0-forms) leading to functions on oriented edges (1-form). 

\paragraph{}
Much effort has gone into estimating the spectral radius of the Kirchhoff Laplacian. It is bounded above
by the spectral radius of the {\bf sign-less Kirchhoff Laplacian} $|K|$ in which one takes the
absolute values for each entry. This matrix is a non-negative matrix in the connected case
has a power $|K|^n$ with all positive entries so that by the Perron-Frobenius theorem, the
maximal eigenvalue is unique. (The Kirchhoff matrix itself of course can have multiple maximal 
eigenvalues like for the case of the complete graph). Also, unlike $K$ which is never 
invertible, $|K|$ can be invertible if $G$ is not bipartite. 
If we treat a graph as a one-dimensional simplicial complex 
(ignoring 2 and higher dimensional simplices in the graph), and denote by $d=F$ 
the exterior derivative of this skeleton complex, then 
$(d+d^*)^2 = K_0 + K_1$, where $K=K_0=d^* d$ is the Kirchhoff matrix and $K_1=d d^*$ is the 
one-form matrix with the same spectral radius, leading to 
\cite{AndersonMorley1985}. Much work has gone in improving this 
\cite{BrualdiHoffmann,Stanley1987,LiZhang1997,Zhang2004,FengLiZhang,Shi2007,LiShiuChan2010}.
We have used that an identity coming from connection matrices \cite{Hydrogen}.

\paragraph{}
We stumbled upon the theorem, when looking for bounds on the {\bf tail distribution} 
$\mu([x,\infty))$ of the {\bf limiting density of states} 
$\mu$ of the Barycentric limit $\lim_{n \to \infty} G_n$ of a 
finite simple graph $G=G_0$, where $G_n$ is the Barycentric 
refinement of $G_{n-1}$ in which the complete sub-graphs 
of $G_{n-1}$ are the vertices and two such vertices in $G_n$ are connected if 
one is contained in the other \cite{KnillBarycentric,KnillBarycentric2}. 
The {\bf potential} $U(z) = \int_0^{\infty} \log(z-w) \; d\mu(w)$ is interesting because
the special value $U(-1)$ measures the exponential growth of {\bf rooted spanning forests},
and the special potential value $U(0)$ 
captures the exponential growth of the {\bf rooted spanning trees}. See \cite{TreeForest}.

\paragraph{}
The spectral picture is that in general, the {\bf pseudo determinant}  \cite{cauchybinet}
${\rm Det}(K)=\prod_{\lambda \neq 0} \lambda$ 
is the number of {\bf rooted spanning trees} in $G$ by 
the {\bf Kirchhoff matrix tree theorem} and 
${\rm det}(1+K)$ is the number of rooted spanning forests in $G$
by the {\bf Chebotarev-Shamis matrix forest theorem} 
\cite{ChebotarevShamis1,ChebotarevShamis2,Knillforest}. All these relations follow directly
from the {\bf generalized Cauchy-Binet theorem} that states
that for any $n \times m$ matrices $F,G$, one has the pseudo determinant version
${\rm Det}(F^T G) = \sum_{|P|=k} \det(F_P) \det(G_P)$ with $k$ depends on $F,G$ and
$\det(1+x F^T G) = \sum_P x^{|P|} \det(F_P) \det(G_P)$.
{\bf Pythagorean identities} like ${\rm Det}(F^T F) = \sum_{|P|=k} \det^2(F_P)$
and $\det(1+F^T F) = \sum_P \det^2(F_P)$ follow for an arbitrary $n \times m$ matrix $F$.
Applied to the {\bf incidence matrix} $F$ of a connected finite simple graph, 
where $k(A)=n-1$ is the rank of $K=F^T F$, the first identity counts on the right hand side
the number of spanning trees and the second identity counts on the right hand side 
the number of spanning forests.  In the quiver case, we must think of ${\rm Det}(K)$ as
some soft of ``virtual tree count": take a {\bf clover graph} $G=\{\{1\},\{ (1,1),(1,1),(1,1)\} \}$,
a single node with $3$ loops, where $K=[3]$ has ${\rm Det}(K)={\rm det}(K)=3$ even so, 
technically, there is only one
``spanning tree", a tree without edges. In the {\bf ribbon multi-graph}
$G=\{\{1,2\},\{ (1,2),(1,2),(1,2)\} \}$, the spectrum of the Kirchhoff matrix is $\{0,6\}$
with  pseudo determinant ${\rm Det}(K)=6$ which 
counts the number of rooted spanning trees and the Fredholm determinant ${\rm det}(1+K)=7$
counts the number of rooted spanning forest. 
Tree and forest theorems work multi-graphs but the concept would need modification for 
quivers that have loops.

\paragraph{}
Having noticed that the {\bf tree-forest ratio} 
$\tau(G)={\rm Det}(1+K)/{\rm Det}(K)=\prod_{\lambda \neq 0} (1+1/\lambda)$ 
has a Barycentric limit $\lim_{n \to \infty} (1/|G_n|) \log(\tau(G_n))$,
we interpreted this as $U(-1)-U(0)$, but this required to see that 
the normalized potential $U(z)$ exists. This required spectral results.
By the way, for {\bf complete graphs} $G=K_n$, the {\bf tree-forest 
ratio} is $(1+1/n)^{n-1}$ and converges to 
the {\bf Euler number} $e$ for $n \to \infty$. For triangle-free
graphs (or considering the 1-dimensional skeleton complex on a graph only), 
$\log(\tau(G_n))/|V(G_n)|$ converges to $\log(\phi^2)$, where
$\phi=(1+\sqrt{5})/2$ is the {\bf golden ratio}.  
For example, for $G=C_n$, where ${\rm Det}(K)=n^2$ and the number of rooted spanning forests,
it is the alternate {\bf Lucas number}, recursively defined by 
$L(n+1) = 3 L(n)-L(n-1)+2, L(0)=0,L(1)=1$.  
We proved in general that $\log(\tau(G_n))/|V(G_n)|$ converges under 
{\bf Barycentric refinements} $G_0 \to G_1 \to G_2 \dots$
for an arbitrary graph $G=G_0$ (with Whitney simplicial complex) to a {\bf universal constant} that 
only depends on the maximal dimension of $G$ and is related to the universal
potential function $U(z)$ in dimension $d$. This is covered in \cite{TreeForest}.

\paragraph{}
Theorem~(\ref{1}) needs to be placed into the context of the spectral graph literature like 
\cite{Biggs,VerdiereGraphSpectra,Godsil,Brouwer,Chung97,Mieghem,DvetkovicDoobSachs,
Spielman2009,Nica2018} 
or articles like \cite{Guo2005,Stephen,GroneMerrisSunder2} as well as work in the higher
dimensional case \cite{DanijelaJost,DanijelaJost2}. 
Most of the research work in this area has focused on small
eigenvalues, like $\lambda_2$ or large 
eigenvalues, like the spectral radius $\lambda_n$. For $\lambda_2$, 
there is a {\bf Cheeger estimate} $h^2/(2d_n) \leq \lambda_2 \leq h$ 
for the eigenvalue $\lambda_2$ and {\bf Cheeger constant}
$h=h(G)$ \cite{Cheeger1970} first defined for Riemannian manifolds, 
meaning in the graph case that one needs remove $h |H|$ edges 
to separate a sub-graph $H$ from $G$ (see \cite{VerdiereGraphSpectra}).
For the largest eigenvalue $\lambda_n$, the Anderson-Morley bound 
has produced an industry of results. 

\paragraph{}
Let us look at some example. Figure~(1) shows more visually
what happens in some examples of graphs with $n=10$ vertices. \\
a) For the {\bf cyclic graph} $C_4$, the Kirchhoff 
eigenvalues are $\lambda_1=0,\lambda_2=2,\lambda_3=2,\lambda_4=4$
and the edge degrees are $d_1=d_2=d_3=d_4=2$. \\
b) For the {\bf star graph} with $n-1$ spikes, the eigenvalues are 
$\lambda_1=0,\lambda_2= \cdots \lambda_{n-1}=1,
\lambda_n=n$ while the degree sequence is 
$d_1=\cdots = d_{n-1}=1, d_n=n$. \\
c) For a {\bf complete bipartite graph} $K_{n,m}$ with $m \leq n$, 
we have $0 \leq m \leq \cdots \leq m \leq n \cdots \leq  n$ with  
$m-1$ eigenvalues $m$ and $n-1$ eigenvalues $n$. The degree sequence
has $m \leq m \leq \cdots \leq m \leq n \leq \cdots \leq n$ with 
$m$ entries $m$ and $n$ entries $n$. 

\paragraph{}
From all the $38$ isomorphism classes of connected finite simple graphs with $4$ vertices, 
there are only $3$, for which equality holds for some $k$ in Theorem~(\ref{1}).
From the $728$ isomorphism classes of connected finite simple graphs with $5$ vertices, 
there are none with equality for Theorem~(\ref{1}).
From the $26704$ isomorphism classes of connected graphs with $6$ vertices, 
there are $70$ for which Theorem~(\ref{1}) has equality for some $k$.
It is always only the largest eigenvalue, where we have seen equality to hold. 

\paragraph{}
The theorem implies $\sum_{j=1}^k \lambda_j \leq 2\sum_{j=1}^k d_j$ which is
weaker than the {\bf Schur-Horn inequality} $\sum_{j=1}^k \lambda_j \leq \sum_{j=1}^k d_j$
\cite{Schur1911}.
The Schur-Horn inequality is of wider interest. It can be seen in the context of partial traces 
\cite{TaoSchurHorn} $\sum_{j=1}^k \lambda_j = {\rm inf}_{dim(V)=k} {\rm tr}(A|V)$ and is
special case of the {\bf Atiyah-Guillemin-Sternberg convexity theorem} 
\cite{Atiyah1982,GuilleminSternberg1982}. The Schur inequality has been 
sharpened a bit for Kirchhoff matrices to 
$\sum_{j=1}^k \lambda_j \leq \sum_{j=1}^k d_j-1$ \cite{Brouwer} 
Proposition 3.10.1. 

\paragraph{}
Unlike the Schur-Horn inequality, Theorem~(\ref{1}) does not extend to general 
symmetric matrices. Already for $A=\left[ \begin{array}{cc} 1 & 3 \\ 3 & 1 \end{array} \right]$
which has eigenvalues $\lambda_1=-2, \lambda_2=4$ and with $d_1=1,d_2=1$, the inequality
$\lambda_2 \leq 2 d_2$ fails. It also does not extend to symmetric matrices
with non-negative eigenvalues, but also this does not work as
$A=\left[ \begin{array}{ccc} 1 & 1 & 1 \\ 1 & 1 & 1 \\ 1 & 1 & 1 \end{array} \right]$ shows, 
as this matrix has eigenvalues $0,0,3$ and diagonal entries $1,1,1$. The case of $n \times n$
matrices with constant entries $1$ is an example with eigenvalues $0$ and $n$ showing that no
estimate $\lambda_n \leq C d_n$ is in general possible for symmetric matrices, even when 
asking the diagonal entry to dominate the other entries. 

\section{Open ends}

\paragraph{}
When we looked at the data which led to Theorem~(1), they first 
indicated to investigate the {\bf Schur-Horn error}
$e_k = \sum_{j=1}^k d_j - \sum_{j=1}^k \lambda_j$. The Schur-Horn 
inequality tells that $e_k \geq 0$.  
We experimentally saw that the sequence $e_k$ has some asymptotics
looking like a parabola for large graphs.
An estimate $e_k \leq d_k$ would with the Schur-Horn inequality give
$0 \leq \sum_{j=1}^k d_j - \sum_{j=1}^k \lambda_j  
+ (d_{k+1}-\lambda_{k+1}) \leq d_k + d_{k-1}+ d_{k+1} - \lambda_{k+1}$ 
and give $\lambda_k \leq d_k+d_{k-1}+d_{k-2}$ for all $k$.
While $e_k \leq d_k$ does not hold in general, one can ask for reasonable 
bounds on the Schur-Horn error $e_k$ and ask what the sequence $e_k$ or 
the total Schur-Horn error $E(G)=\sum_k e_k$ tells about the graph. 
For which graphs with $n$ vertices is $E(G)$ minimal or maximal? 

\paragraph{}
The Brouwer-Haemers bound $\lambda_k \geq {\rm max}(0,d_k-(n-k))$
is good for large $k$ but far from optimal for lower energies, especially if $n$
is large. We would definitely like to have better general lower bounds. 
\footnote{Part of the graph theory literature labels the eigenvalues 
in decreasing order. We use an ordering more familiar in the manifold case, 
where one has no largest eigenvalue and which also appears in the earlier 
literature, like in \cite{HornJohnson2012}.}
We have seen that the bound $\lambda_1 \geq d_1-(n-1)$ already does not hold 
in general for quivers with multiple-connections. Ribbon graphs were 
counter examples. Theorem~2 required to have 
no multiple connections. We believe that the original Brouwer-Haemers bound 
$\lambda_k \geq d_k-(n-k)-1$ (with the additional $-1$) could hold for 
{\bf every quiver and all $k \geq 2$}, 
similarly as in the finite simple graph case, where also the case $k=1$ is an exception. 
We have seen that for the ribbon graph and $k=1$, the discrepancy $(d_1-(n-1))-\lambda_1$ 
can be arbitrarily large, so that also the subtraction of $1$ did not help. 
But in experiments so far it appears only to be an exception for the lowest eigenvalue 
$\lambda_1$. 

\paragraph{}
We see for many {\bf Erd\"os-R\'enyi graphs} (finite simple graphs with $n$ 
vertices, where an edge is turned on with probability $p$) that an upper bound 
$\lambda_k \leq C d_k$ holds for most graphs
for any $C>1$, if the number of vertices is large. Theorem~1 shows that
$C=2$ works as $\lambda_k \leq d_k+d_{k-1} \leq 2 d_k$. 
One can even ask for a linear or almost linear bound in the Erdoes-Renyi case:
for all $C>1$ and all $p \in [0,1]$, we expect the probability of the
set of graphs in the Erd\"os-R\'enyi probability space 
$E(n,p)$ with $\lambda_k \leq C d_k$ for all $1 \leq k \leq n$ 
to go to $1$ as $n \to \infty$. There are several ways to produce probability
spaces on quivers that are analog to Erd\"os-R\'enyi spaces. For quivers, 
the number $m$ of edges is not bound by the number $n$ of vertices like for finite
simple graphs (there is no Kruskal-Cotana bound). One could start with 
Erd\"os-R\'enyi and use geometric distributions to 
add multiple edges loops. A more elegant probability space on quivers is to fix $V$ and 
look at all maps $V \to V \times V$ with uniform distribution and chose $m \geq 0$ random 
independent samples $E=\{ (f_k,g_k), k \leq m\}$, where $m$ has a geometric distribution 
${\rm P}[m=k]=p^{k}(1-p)$. The geometric distribution is the most obvious because it 
maximizes entropy on all probability distributions on $\mathbb{N}$ similarly as the exponential
does on $[0,\infty)$ or the Gaussian does on $\mathbb{R}$. 

\begin{figure}[!htpb]
\scalebox{0.5}{\includegraphics{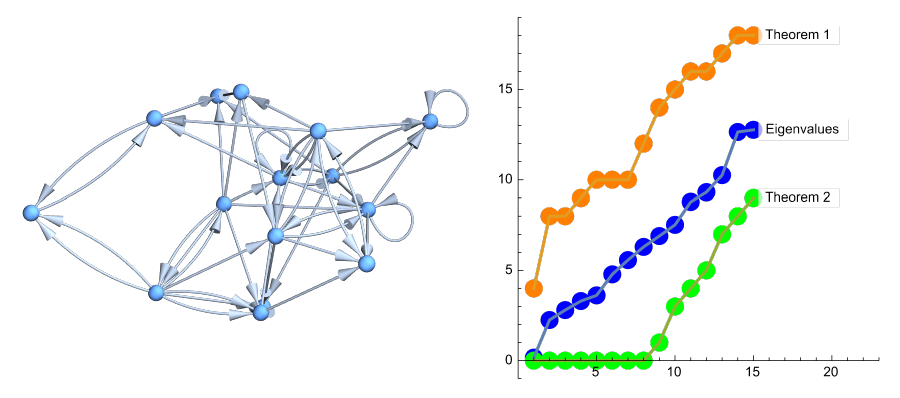}}
\label{Spectra}
\caption{
We see a random quiver in which the spectrum is compared with the
upper bound and lower bound according to Theorem~1 and Theorem~2. 
The lower bound does not hold for all quivers, as Theorem~2 requires 
to have no multiple connections. 
}
\end{figure}

\section{Illustration}

\paragraph{}
Here are a few examples of spectra with known upper and 
lower bounds:

\begin{figure}[!htpb]
\scalebox{0.35}{\includegraphics{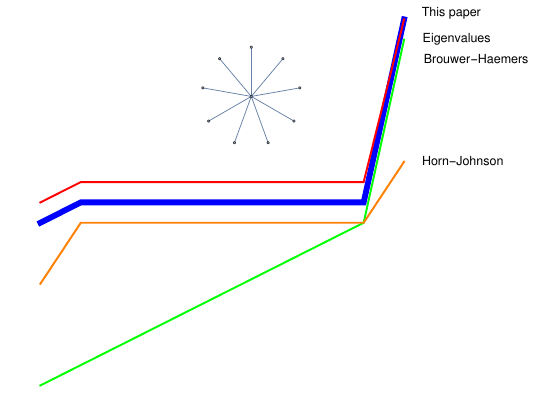}}
\scalebox{0.35}{\includegraphics{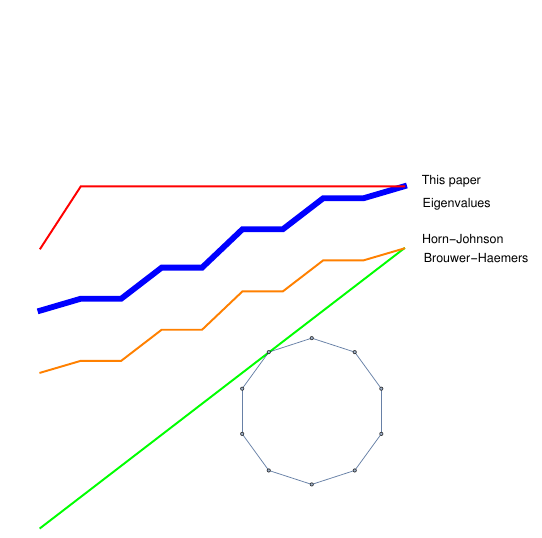}}
\scalebox{0.35}{\includegraphics{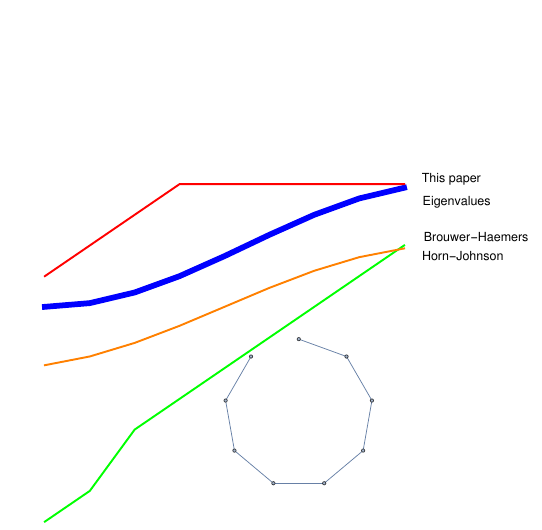}}
\scalebox{0.35}{\includegraphics{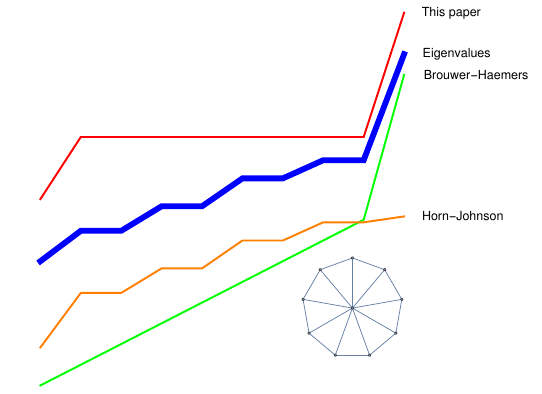}}
\scalebox{0.35}{\includegraphics{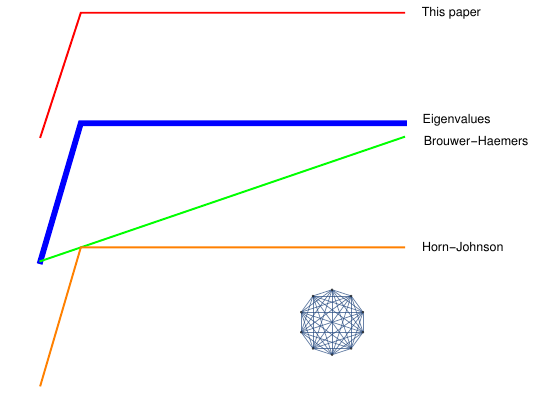}}
\scalebox{0.35}{\includegraphics{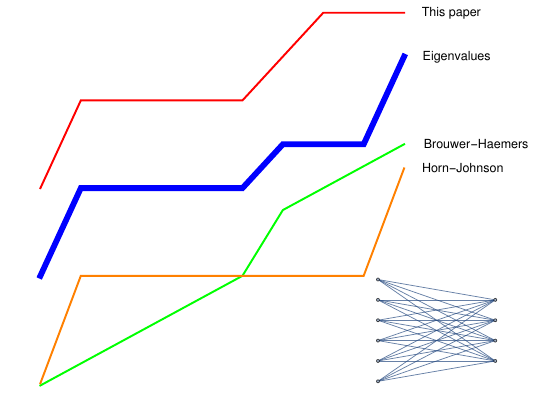}}
\scalebox{0.35}{\includegraphics{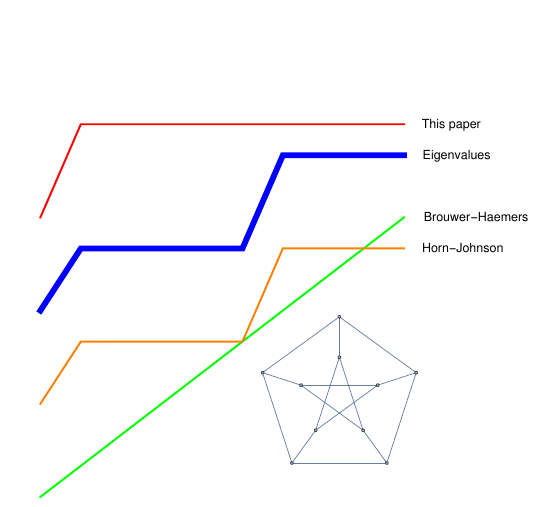}}
\scalebox{0.35}{\includegraphics{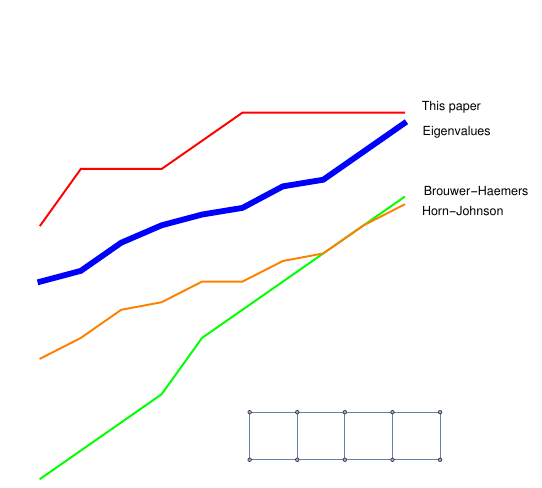}}
\scalebox{0.35}{\includegraphics{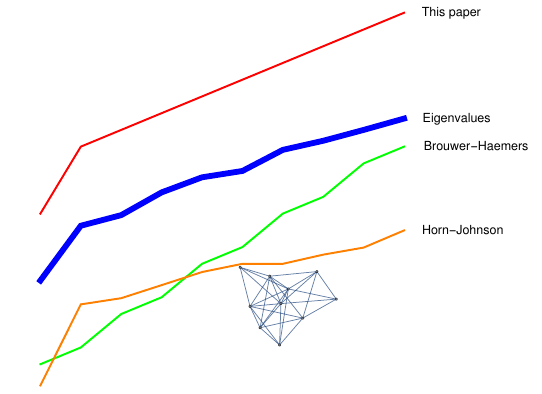}}
\scalebox{0.35}{\includegraphics{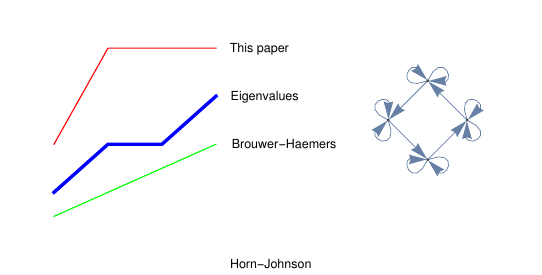}}
\label{Spectra}
\caption{
This figure shows examples of spectra of simple graphs and compares them
with known upper and lower bounds. 
We see first the {\bf Star, Cycle and Path graph}, then 
the {\bf Wheel, Complete, Bipartite graph},
and finally the {\bf Petersen, Grid and Random graph},  
all with $10$ vertices. The last is a graph with loops. 
The eigenvalues are outlined thick.
The graph above the spectrum is upper bound of Theorem~1.
The lower bound of Theorem~2 and the Horn-Johnson 
bounds are below. They can complement each other.
}
\end{figure}

\section{Code}

\paragraph{}
We add some Mathematica procedures illustrating Theorem~1 and Theorem~2 and allowing
a reader to experiment. One can for example generate random quivers for a number of
vertices $n$ and a number of edges $m$, the number of loops $l$ and the number of 
multiple connections $q$. We see as examples
the Goodwill hunting graph and the Koenigsberg graph \cite{Euler1736}. 
We had to program the Quiver Kirchhoff matrix procedure because Mathematica gives for 
an undirected multi-graph with loops the Kirchhoff matrix of the corresponding finite 
simple graph in which all multi connections and loops are discarded. 
We also check that if we take the Quiver Kirchhoff matrix of a
finite simple graph, we get the usual Kirchhoff matrix already implemented
in the Wolfram language. 

\begin{tiny}
\lstset{language=Mathematica} \lstset{frameround=fttt}
\begin{lstlisting}[frame=single]
QuiverGradient[s_]:=Module[{v=VertexList[s],e=EdgeList[s],n,m,F},
 n=Length[v];m=Length[e];    F=Table[0,{m},{n}];
 Do[{a,b}={e[[k,1]],e[[k,2]]};F[[k,a]]+=1;If[a!=b,F[[k,b]]+=-1],{k,m}];F];
QuiverKirchhoffMatrix[s_]:=Module[{F=QuiverGradient[s]},Transpose[F].F];
QuiverOneFormMatrix[s_]:=Module[{F=QuiverGradient[s]},F.Transpose[F]];
RandomQuiver[{n_,m_,l_,c_}]:=Module[{G=RandomGraph[{n,m}],v,e,q={},R,A},
 v=VertexList[G];e=EdgeRules[G];R=RandomChoice;A=Append;U[u_]:=u[[1]]->u[[2]];
 Do[x=R[v];q=A[q,x->x],{l}];Do[q=A[q,U[R[e]]],{c}];Graph[v,Join[e,q]]];
RandomQuiver[{n_,m_}]:=Module[{v=Range[n],e},
  e=Table[RandomChoice[v]->RandomChoice[v],{m}]; Graph[v,e]];
ErdoesRenyi[{n_,p_}]:=RandomQuiver[{n,Random[GeometricDistribution[p]]}];
SpectralComparison[G_]:=Module[{K=QuiverKirchhoffMatrix[G],n},n=Length[K]; 
Eigen = Sort[Eigenvalues[1.0*K]];  d=Sort[Table[K[[k,k]],{k,n}]]; 
Upper = Table[d[[k]]+If[k>1,d[[k-1]],0],{k,n}];
Lower = Table[Max[0,d[[k]]-(n-k)], {k,n}];
S1=ListPlot[{Lower,Eigen,Upper},Filling->Bottom];
S2=ListPlot[{Lower,Eigen,Upper},Filling->Bottom,Joined->True];
Show[{S1,S2},Ticks->None,PlotRange->{{1,n},{0,Max[Upper]}}]];
EulerChar[G_]:=Length[VertexList[G]]-Length[EdgeList[G]]; 
b0[G_]:=Length[NullSpace[QuiverKirchhoffMatrix[G]]];
b1[G_]:=Length[NullSpace[QuiverOneFormMatrix[G]]];
Cohomology[G_]:=b0[G]-b1[G]; Betti[G_]:={b0[G],b1[G]};
Fvector[G_]:={Length[VertexList[G]],Length[EdgeList[G]]}; 
G=RandomQuiver[{25,205,30,40}]; 
Print["Euler-Poincare Check: ",EulerChar[G]==Cohomology[G]];
G=UndirectedGraph[RandomQuiver[{15,35,0,0}]];
Print["Matrix : ",QuiverKirchhoffMatrix[G]==Normal[KirchhoffMatrix[G]]];
G = RandomQuiver[{105,1005,300,2000}];         U1=SpectralComparison[G];
G =Graph[{1->2,2->4,4->1,2->3,2->3}];          U2=SpectralComparison[G]; 
Print["GoodWillHunting: ","Betti=",Betti[G],"  Fvector=",Fvector[G]]; 
G =Graph[{1->2,2->1,4->1,1->4,2->3,3->4,1->3}];U3=SpectralComparison[G];
Print["Euler Koenigsberg: ","Betti=",Betti[G],"  Fvector=",Fvector[G]]; 
Print["Theorem Illustration: "]; GraphicsRow[{U1,U2,U3}]
GraphicsRow[Table[GraphPlot3D[RandomQuiver[{10^k,1000}]],{k,1,3}]]
\end{lstlisting}
\end{tiny}

\begin{figure}[!htpb]
\scalebox{0.6}{\includegraphics{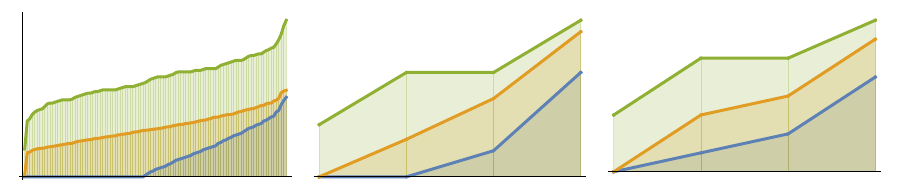}}
\label{Code}
\caption{
This is the output graphics of the above code. The eigenvalue graph
is sandwiched in general. The upper bound holds for all 
quivers, the lower bound only for quivers without multiple
connections.  The first case shows a random quiver, the 
second the Good-Will-Hunting multi-graph with Betti
vector $b=(1,2)$ and $f$-vector $f=(4,5)$, the third is
the K\"onigsberg multi graph with $b=(1,4)$ and $f=(4,7)$. 
}
\end{figure}

\begin{figure}[!htpb]
\scalebox{0.6}{\includegraphics{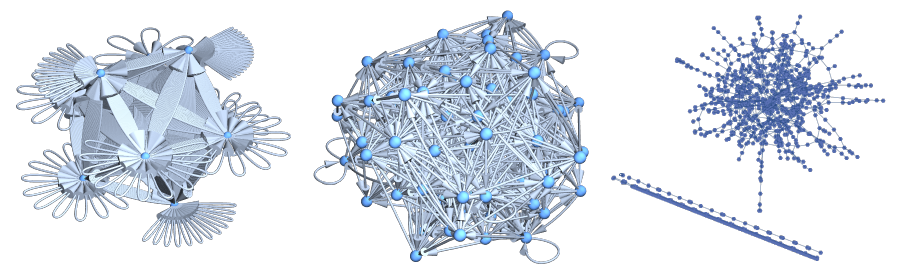}}
\label{erdoes-Renyi}
\caption{
Three quiver Erdoes-Renyi choices with $m=1000$ random edges
$a \to b$, where $a,b$ are randomly chosen vertices and each
$(a,b)$ produces either an edge ($a \neq b$) or a loop ($a=b$). 
The first has $n=10$ nodes, the second $n=100$ nodes,
the third $n=1000$, keeping $m=1000$ edges.
}
\end{figure}

\bibliographystyle{plain}

\end{document}